\documentclass[11pt]{amsart}
\usepackage{amsmath, amsthm, amssymb,amscd}
\usepackage{graphicx}
\usepackage[english]{babel}
\newcommand{\F}{\mathcal{F}}

\newtheorem{thm}{Theorem}[section]
\newtheorem*{thm-nl}{Theorem}
\newtheorem*{prop-nl}{Proposition}
\newtheorem{definition}[thm]{Definition}

\def\PP{{\textbf P}}
\def\OO{\mathcal{O}}

\def\F{\mathcal{F}}

\def\cM{\mathcal{M}}
\def\cR{\mathcal{R}}

\def\Pic0{{\rm Pic}^0(X)}

\def\mm{\overline{\mathcal{M}}}

\newtheorem*{cor-nl}{Corollary}
\newtheorem{conjecture}[thm]{Conjecture}
\newtheorem*{conjecture-nl}{Conjecture}
\newtheorem*{quest-nl}{Question}
\newtheorem*{quests-nl}{Questions}

\newtheorem{prop}[thm]{Proposition}

\theoremstyle{remark}

\newtheorem{remark}[thm]{Remark}

\title{{Progress on syzygies of algebraic curves}}

\author[G. Farkas]{Gavril Farkas}
\address{Humboldt-Universit\"at zu Berlin, Institut f\"ur Mathematik,  Unter den Linden 6
\hfill \newline\texttt{}
 \indent 10099 Berlin, Germany} \email{{\tt farkas@math.hu-berlin.de}}

\begin{document}

\begin{abstract}
These notes discuss recent advances on syzygies on algebraic curves, especially concerning the Green, the Prym-Green and the Green-Lazarsfeld Secant Conjectures.
The methods used are largely geometric and variational, with  a special emphasis on examples and explicit calculations. The notes are based on series of lectures
given in Daejeon (March 2013), Rome (November-December 2015) and Guanajuato (February 2016).
\end{abstract}

\maketitle
\section{Introduction}

The terms \emph{syzygy} was originally used in astronomy and refers to three celestial bodies (for instance Earth-Sun-Moon) lying on a straight line. In mathematics, the term was introduced in 1850 by
Sylvester, one of the greatest mathematical innovators and neologists \footnote{Sylvester is responsible for a remarkable number of standard mathematical terms like matrix, discriminant,
minor, Jacobian, Hessian, invariant, covariant and many others. Some of his other terms have not stuck, for instance his  \emph{derogatory matrices}, that is, matrices whose characteristic polynomial
differs from the minimal polynomial are all but forgotten today.} of all times. For Sylvester a syzygy is a linear relation between certain functions with arbitrary functional coefficients,
which he called syzygetic multipliers. We quote from Sylvester's original paper in the Cambridge and Dublin Journal of Mathematics \textbf{5} (1850), page 276: \emph{The members of any
group of functions, more than two in number, whose nullity is implied in the relation of double contact, whether such group form a complete system or not, must be
in syzygy}.

\vskip 3pt

Most of the original applications, due to Cayley, Sylvester and others involved Classical Invariant Theory, where the syzygies in question were algebraic relations between
the invariants and covariants of binary forms. An illustrative example in this sense is represented by the case of a general binary cubic form $f(x,y)=ax^3+3bx^2y+3cxy^2+dy^3$.
One can show that there exists a covariant $T$ of degree $3$ and order $3$ (that is, of bidegree $(3,3)$ in the coefficients of $f$ and in the
variables $x$ and $y$ respectively), having the rather unforgiving form
$$T:=(a^2d-3abc+2b^3)x^3+3(abd+b^2c-2ac^2)x^2y-3(acd-2b^2d+bc^2)xy^2-(ad^2-3bcd+2c^3)y^3.$$
Denoting by $D$ and by $H$ the discriminant and the Hessian of the form $f$ respectively, these three covariants are related by the following \emph{syzygy} of order $6$ and degree $6$:
$$4H^3=Df^2-T^2.$$
The mathematical literature of the $19$th century is full of results and methods of finding  syzygies between invariants of binary forms, and sophisticated algorithms, often based on experience
and intuition rather than on solid proofs, have been devised.

\vskip 3pt

Sylvester pursued an ill-fated attempt to unite mathematics and poetry, which he deemed both to be guided by comparable concerns with formal relations between quantities.
In his treatise \emph{Laws of verse}, he
even introduced the concept of \emph{phonetic syzygy} as a repetition of syllables in certain rhymes and wrote poems to illustrate the principle of
poetic syzygy. Not surprisingly, his ideas and terminology in this direction have not become widespread.

\vskip 3pt

Returning to mathematics, it was Hilbert's landmark paper \cite{Hil} from 1890 that not only put an end to Classical Invariant Theory in its constructive form propagated by the German and the British schools, but also introduced syzygies as objects of pure algebra. \emph{Hilbert's Syzygy Theorem} led to a new world of free resolutions, higher syzygies and  homological algebra.
Although Hilbert's original motivation was Invariant Theory, his ideas had immediate and widespread impact, influencing the entire development of commutative algebra and algebraic geometry.

\vskip 3pt

In algebraic geometry, the first forays into syzygies of algebraic varieties came from two different directions. In Germany, Brill and M. Noether pursued a  long-standing program of bringing algebra into the realm of Riemann surfaces and thus making Riemann's work rigorous. Although the curves Brill and Noether were primarily concerned with were plane curves with ordinary singularities, they had a profound understanding of the importance of the canonical linear system of a curve and raised for the first time the question of describing canonical curves by algebraic equations. In Italy, in 1893 Castelnuovo's \cite{Ca} using purely geometric methods which he also employed in the proof of his bound on the degree of a curve in projective space,  showed that for a curve
$C\subseteq \PP^3$ of degree $d$, hypersurfaces of degree $n\geq d-2$ cut out the complete linear system $|\OO_C(n)|$. At this point, we mention the work of Petri \cite{P} (a student of Lindemann, who became a teacher but remained under the strong influence of Max and Emmy Noether). It revisits a topic that had already been considered by Enriques in 1919 and it gives a complete proof on the presentation of the generators of the ideal of a canonical curve using the algebraic methods of Brill, Noether (and Hilbert). In Petri's work, whose importance would only be recognized much later with the advent of modern Brill-Noether theory, the structure of the equations of canonical curves come to the very center of investigation.

\vskip 4pt

Once the foundations of algebraic geometry had been rigorously laid out,  Serre's sheaf cohomology had been developed  and people could return to the central problems of the subject, the idea of using homological algebra in order to study systematically the geometry of projective varieties can be traced back at least to Grothendieck and Mumford. In 1966 Mumford introduced a fundamental homological invariant, the \emph{Castelunovo-Mumford regularity}, in order to describe qualitatively the equations of an algebraic variety. He gave a fundamental bound for this invariant in terms of the degree and recovered in this way Castelnuovo's classical result \cite{Ca} for curves.  Grothendieck's construction \cite{Gr} of Hilbert schemes parametrizing all subschemes $X\subseteq \PP^n$ having a fixed Hilbert polynomial relies on the possibility of effectively bounding the degree of all equations of a variety with fixed Hilbert polynomial.
Syzygies per se however became mainstream in algebraic geometry
only after  Green \cite{G} introduced Koszul cohomology and repackaged in ways appealing
to algebraic geometers all the information contained in the minimal free resolution of the coordinate ring of an algebraic variety. Striking new relationships between
free resolutions on one side and moduli spaces on the other have been found. For instance, matrix factorizations discovered in algebraic context by Eisenbud turned out to fit
into the framework of $A_{\infty}$-algebras, and as such had recent important applications to mirror symmetry and enumerative geometry.

\section{Syzygies of graded modules over polynomial algebras}

We fix the polynomial ring $S:=\mathbb C[x_0, \ldots, x_r]$ in $r+1$ variables and let $M=\oplus_{d\geq 0} M_d$ be a finitely generated graded $S$-module. Choose a minimal set of
(homogeneous) generators $(m_1, \ldots, m_t)$ of $M$, where $m_i$ is an element of $M$ of degree $a_i$ for $i=1, \ldots, t$. We denote by $K_1$ the module of relations between the
elements $m_i$, that is, defined by the exact sequence
$$0\longleftarrow M\stackrel{(m_i)}\longleftarrow \bigoplus_{i=1}^t S(-a_i)\longleftarrow K_1\longleftarrow 0.$$
The module  $K_1$ is a submodule of a finitely generated module, thus by the Hilbert Basis Theorem, it is a finitely generated graded $S$-module itself. Its degree $d$ piece consists of
$t$-tuples of homogeneous polynomials $(f_1, \ldots, f_t)$,
where $\mbox{deg}(f_i)=d-a_i$ such that
$$\sum_{i=1}^t f_im_i=0\in M_d.$$ Elements of $K_1$ are called first order syzygies of $M$.
Let us choose a minimal set of generators of $K_1$ consisting of relations
$(R_1, \ldots, R_s)$, where $$R_j:=(f_1^{(j)},\ldots, f_t^{(j)}), \mbox{ with } f_1^{(j)}m_1+\cdots+f_t^{(j)}m_r=0,$$
for $j=1, \ldots, s$. Here $f_i^{(j)}$ is a homogeneous polynomial of degree $b_j-a_i$ for $i=1, \ldots, t$ and $j=1, \ldots, s$.
Thus we have an induced map of \emph{free} $S$-modules
$$F_0:=\bigoplus_{i=1}^t S(-a_i)\stackrel{(f_i^{(j)})}\longleftarrow \bigoplus_{j=1}^s S(-b_j)=:F_1.$$ Then we move on and resolve $K_1$, to find a minimal set of relations
among the relations between the generators of $M$, that is, we consider the finitely generated $S$-module $K_2$ defined by the following exact sequence:
$$0\longleftarrow K_1\stackrel{(R_j)}\longleftarrow \bigoplus_{j=1}^s S(-b_j)\longleftarrow K_2\longleftarrow 0.$$
Elements of $K_2$ are relations between the relations of the generators of $M$ and as such, they are called second order syzygies of $M$. One can now continue and resolve the finitely generated
$S$-module $K_2$. The fact that this process terminates after at most $r+1$ steps is the content
of \emph{Hilbert's Syzygy Theorem}, see for instance \cite{Ei} Theorem 1.1:

\begin{thm} Every finitely generated graded $S$-module $M$ admits a minimal free graded $S$-resolution
 $$\mathbb F_{\bullet}: 0\longleftarrow M\longleftarrow F_{0}\longleftarrow F_1\longleftarrow \cdots \longleftarrow F_r\longleftarrow F_{r+1}\longleftarrow 0.$$
\end{thm}

The minimal free resolution $\mathbb F_{\bullet}$ is uniquely determined up to an isomorphism of complexes of free $S$-modules. In particular, each two resolutions have the same length.
Every individual piece $F_p$ of the resolution is uniquely determined, as a graded module, by its numbers of generators and their degrees and  can be written as
$$F_p:=\bigoplus_{q>0} S(-p-q)^{\oplus b_{pq}(M)}.$$
The quantities $b_{p,q}(M)$ have an intrinsic meaning and depend only on $M$. In fact, by the very definition of the $\mbox{Tor}$ functor, one has
$$b_{p,q}(M)=\mbox{dim}_{\mathbb C} \mbox{Tor}^p(M,\mathbb C)_{p+q}, $$
see also \cite{Ei} Proposition 1.7. There is a convenient way of packaging together the numerical information contained in the resolution
$\mathbb F_{\bullet}$, which due to the computer algebra software system \emph{Macaulay}, has become widespread:

\begin{definition} The graded Betti diagram of the $S$-module $M$ is obtained by placing in column $p$ and row $q$ the Betti number $b_{p,q}(M)$.
\end{definition}

Thus the column $p$ of the Betti diagram encoded the number of  $p$-th syzygies of $M$ of various weights. Since it is customary to write the columns of a table from left
to right, it is for this reason that the rows in the resolution $\mathbb F_{\bullet}$ go from right to left, which requires some getting used to.

\vskip 4pt

A much coarser invariant of the $S$-module $M$ than the Betti diagram is its \emph{Hilbert function} $h_M:\mathbb Z\rightarrow \mathbb Z$, given by $h_M(d):=\mbox{dim}_{\mathbb C}(M_d)$.
The Betti diagram determines the Hilbert function of $M$ via the following formula:
$$h_M(d)=\sum_{p\geq 0} (-1)^p\mbox{dim}_{\mathbb C} F_p(d)=\sum_{p\geq 0,q>0} (-1)^p b_{p,q}(M) {d+r-p-q\choose r}.$$
Conversely, the Hilbert function of $M$ determines the alternating sum of Betti numbers on each diagonal of the Betti diagram. For fixed $k\geq 0$, we denote by
$B_k:=\sum_p (-1)^p b_{p,k-p}(M)$ the corresponding alternating sum of Betti numbers in one of the diagonals of the Betti diagram of $M$. The quantities $B_k$ can then
be determined inductively from the Hilbert function, using the formula:
$$B_k=h_M(k)-\sum_{\ell<k} B_{\ell}{r+k-\ell\choose r}.$$
In algebro-geometric applications, the alternating sum of Betti numbers on diagonals correspond to the geometric constraints of the problem at hand. The central question in
syzygy theory is thus to determine the possible Betti diagrams corresponding to a given Hilbert function.

\vskip 4pt

In order to explicitly compute the Betti numbers $b_{p,q}(M)$ it is useful to remember that the $\mbox{Tor}$ functor is symmetric in its two variables. In particular,  there exists a canonical
isomorphisms $\mbox{Tor}^p(M,\mathbb C)\cong \mbox{Tor}^p(\mathbb C, M)$. To compute the last $\mbox{Tor}$ group one is thus led to take an explicit  resolution of the $S$-module $\mathbb C$
by free graded $S$-modules. This is given by the \emph{Koszul complex}. We denote by $V:=S_1=\mathbb C[x_0, \ldots, x_r]_1$ the vector space of linear polynomials in $r+1$ variables.

\begin{thm}\label{koszulcomplex}
 The minimal free $S$-resolution of the module $\mathbb C$ is computed by the Koszul complex in $r+1$ variables:
 $$0\rightarrow \bigwedge^{r+1} V\otimes S(-r-1)\rightarrow \cdots \rightarrow \bigwedge^p V\otimes S(-p)\rightarrow \bigwedge^{p-1}V\otimes S(-p+1)\rightarrow \cdots
 \rightarrow V\otimes S(-1)\rightarrow \mathbb C\rightarrow 0.$$
The $p$-th differential in degree $p+q$ of this complex, denoted by $d_{p,q}:\bigwedge^p V\otimes S_q\rightarrow \bigwedge^{p-1} V\otimes S_{q+1}$, is
 given by the following formula
 $$d_{p,q}(f_1\wedge \ldots \wedge f_p\otimes u)=\sum_{\ell=1}^p (-1)^{\ell} f_1\wedge \ldots \wedge \hat{f_{\ell}}\wedge \ldots \ \wedge f_p\otimes (uf_{\ell}),$$
where $f_1, \ldots, f_p\in V$ and $u\in S_q$.
\end{thm}

In order to compute the Betti numbers of $M$, we tensor the Koszul complex with the $S$-module $M$ and take cohomology. One is thus naturally led to the definition
of \emph{Koszul cohomology} of $M$, due to Green \cite{G}. Even though Green's repackaging of the higher Tor functors amounted to little new information, the
importance of \cite{G} cannot be overstated, for it brought syzygies in the realm of mainstream algebraic geometry. For integers $p$ and $q$, one  defines the Koszul cohomology group
$K_{p,q}(M,V)$ to be the cohomology of the  complex
$$\bigwedge^{p+1} V\otimes M_{q-1}\stackrel{d_{p+1,q-1}}\longrightarrow \bigwedge^p V\otimes M_q\stackrel{d_{p,q}}\longrightarrow \bigwedge^{p-1}V\otimes M_{q+1}.$$
As already pointed out, one has
$$b_{p,q}(M)=\mbox{dim}_{\mathbb C} K_{p,q}(M,V).$$

From the definition it follows that Koszul cohomology is functorial. If $f:A\rightarrow B$ is a morphism of graded $S$-modules, one has an induced morphism
$$f_*:K_{p,q}(A,V)\rightarrow K_{p,q}(B,V)$$
of Koszul cohomology groups. More generally, if
$$0\longrightarrow A\longrightarrow B\longrightarrow C\longrightarrow 0$$
is a short exact sequence of graded $S$-modules, one has an associated long exact sequence in Koszul cohomology:
\begin{equation}\label{lonseq}
\cdots \rightarrow K_{p,q}(A,V)\rightarrow K_{p,q}(B,V)\rightarrow K_{p,q}(C,V)\rightarrow K_{p-1,q+1}(A,V)\rightarrow K_{p-1,q+1}(B,V)\rightarrow \cdots
\end{equation}

\section{Syzygies in algebraic geometry}

In algebraic geometry, one is primarily interested in resolving  (twisted) coordinate rings of projective algebraic varieties.  A very good general reference for Koszul cohomology in
algebraic geometry is the book of Aprodu and Nagel \cite{AN}.

\vskip 3pt

We begin by setting notation. Let $X$ be a projective variety, $L$ a globally generated line bundle on $X$ and $\F$ a sheaf on $X$. Set $r=r(L)=h^0(X,L)-1$ and denote by
$\varphi_L:X\rightarrow \PP^r$ the morphism induced by the linear system
$|L|$. Although strictly speaking this is not necessary, let us assume that $L$ is very ample, therefore $\varphi_L$ s an embedding. We set
$S:=\mbox{Sym } H^0(X,L)\cong \mathbb C[x_0, \ldots, x_r]$ and form the twisted coordinate $S$-module
$$\Gamma_X(\F,L):=\bigoplus_{q} H^0(X,\F\otimes L^{\otimes q}).$$

Following notation of Green's \cite{G}, one introduces the Koszul cohomology groups
$$K_{p,q}(X,\F,L):=K_{p,q}\Bigl(\Gamma_X(\F,L),H^0(X,L)\Bigr)$$ and accordingly, one defines the Betti numbers
$$b_{p,q}(X,\F,L):=b_{p,q}\bigl(\Gamma_X(\F,L)\bigr).$$
In most geometric applications, one has $\F=\OO_X$, in which case $\Gamma_X(L):=\Gamma_X(\OO_X,L)$ is the coordinate ring of the variety $X$ under the map $\varphi_L$. One writes
$b_{p,q}(X,L):=b_{p,q}(X, \OO_X,L)$.
It turns out that the calculation of Koszul cohomology groups of line bundles can be reduced to usual cohomology of the exterior powers of a certain vector bundle on the variety $X$.

\begin{definition}
For a globally generated line bundle $L$ on $X$, we define  the \emph{Lazarsfeld vector bundle} $M_L$ via the exact sequence
$$0\longrightarrow M_{L}\longrightarrow H^0(X,L) \otimes \mathcal{O}_X \longrightarrow L\longrightarrow 0,$$ where the above map is given by evaluation of the global sections of $L$.
\end{definition}
One also denotes by $Q_L:=M_L^{\vee}$ the dual of the Lazarsfeld bundle. Note that we have a canonical injection $H^0(X,L)^{\vee}\hookrightarrow H^0(X,Q_L)$ obtained by dualizing the defining
sequence for $M_L$.
To make the role of Lazarsfeld bundles more transparent, we recall the description of the tangent bundle of the projective space $\PP^r$ provided by the \emph{Euler sequence} (see \cite{Ha} Example 8.20.1):
$$0\longrightarrow \OO_{\PP^r}\longrightarrow H^0(\PP^r,\OO_{\PP^r}(1))^{\vee}\otimes \OO_{\PP^r}(1)\longrightarrow T_{\PP^r}\longrightarrow 0,$$
By pulling-back the Euler sequence via the map $\varphi_L$ and dualizing, we observe that the Lazarsfeld  bundle is a twist of the restricted cotangent bundle:
$$M_L\cong \Omega_{\PP^r}(1)\otimes \OO_X.$$ We now take  exterior powers in the exact sequence  defining $M_L$, to  obtain the following exact sequences for each $p\geq 1$:
$$0\longrightarrow \bigwedge^p M_L\longrightarrow \bigwedge^p H^0(X,L)\otimes \OO_X\longrightarrow \bigwedge^{p-1} M_L\otimes L\longrightarrow 0.$$
After tensoring and taking cohomology, we link these sequences to the Koszul complex computing $K_{p,q}(X,\F,L)$ and  obtain the following description of Koszul cohomology groups in
terms of ordinary cohomology of powers
of twisted Lazarsfeld bundles. For a complete proof we refer to \cite{AN} Proposition 2.5:

\begin{prop}\label{kernelbundles}
One has the following canonical isomorphisms:
$$K_{p,q}(X, \mathcal{F}, L)  \cong \mathrm{Coker} \Bigl\{\bigwedge^{p+1}H^0(X,L) \otimes H^0(X,\mathcal{F} \otimes L^{q-1}) \to H^0\bigl(X,\bigwedge^p M_L \otimes \mathcal{F} \otimes L^q
\bigr)\Bigr\}
$$
$$ \cong \mathrm{Ker}\Bigl\{H^1\bigl(X,\bigwedge^{p+1} M_L \otimes \mathcal{F} \otimes L^{q-1}\bigr) \to \bigwedge^{p+1}H^0(X, L) \otimes H^1(X,\mathcal{F} \otimes L^{q-1})\Bigr\}.$$
\end{prop}

The description of Koszul cohomology given in Proposition \ref{kernelbundles} brings syzygy theory firmly in the realm of algebraic geometry, for usual
vector bundle techniques involving stability and geometric constructions  come to the fore in order to compute Koszul cohomology groups.
We now give some examples and to keep things intutive, let us assume $\F=\OO_X$. Then
$$K_{1,1}(X,L)\cong H^0(X,M_L\otimes L)/\bigwedge^2 H^0(X,L)\cong I_2(X,L)$$
is the space of quadrics containing the image of the map $\varphi_L$. Similarly, the group
$$K_{0,2}(X,L)\cong H^0(X,L^2)/\mbox{Sym}^2 H^0(X,L)$$
measures the failure of $X$ to be quadratically normal. More generally, if
$$I_X:=\bigoplus_{q\geq 2} I_X(q)\subseteq S$$ is the graded ideal of $X$, then one has
the following isomorphism, cf. \cite{AN} Proposition 2.8:

$$K_{p,1}(X,L)\cong K_{p-1,2}(\PP^r, I_X,\OO_{\PP^r}(1)).$$
In particular, $K_{2,1}(X,L)\cong \mbox{Ker}\{I_X(2)\otimes H^0(X,\OO_X(1))\rightarrow I_X(3)\}$.

\vskip 4pt

Koszul cohomology shares many of the features of a cohomology theory. We single out one aspect which will play a role later in these lectures:

\begin{thm}\label{lef}
 Koszul cohomology satisfies the Lefschetz hyperplane principle. If $X$ is a projective variety and $L\in \mathrm{Pic}(X)$, assuming $H^1(X,L^{\otimes q})=0$ for all $q$, then for any divisor
  $D\in |L|$ one has an isomorphism $K_{p,q}(X,L)\cong K_{p,q}(D,L_{|D})$.
\end{thm}

In general, it is not easy to determine the syzygies of any variety by direct methods, using just the definition. One of the few instances when this is possible is given by the
twisted cubic curve $R\subseteq \PP^3$. Denoting the coordinates in $\PP^3$ by $x_0, x_1, x_2$ and $x_3$, the ideal of $R$ is given by the $2\times 2$-minors of the
following matrix:
$$\begin{pmatrix}
x_0 & x_1& x_2\\
x_1 & x_2 & x_3\\
\end{pmatrix}
$$
Therefore the ideal of $R$ is generated by three quadratic equations
$$q_1:=x_0x_2-x_1^2, \ q_2:=x_0x_3-x_1x_2 \ \mbox{ and } \ q_3:=x_1x_3-x_2^2,$$
among which there exists two linear syzygies, that is, syzygies with linear coefficients:
$$R_1:=x_0 q_3-x_1q_2+x_2 q_1 \ \mbox{ and } \ R_2:=x_1q_3-x_2q_2+x_3q_1.$$
The resolution of the twisted cubic curve is therefore the following:
$$0\longleftarrow \Gamma_R\bigl(\OO_R(1)\bigr)\longleftarrow S\longleftarrow S(-2)^{\oplus 3}\longleftarrow S(-3)^{\oplus 2}\longleftarrow 0,$$ and the corresponding Betti diagram is
the following (the entries left open being zero):

\begin{table}[htp!]
\begin{center}
\begin{tabular}{|c|c|c|}
\hline
$1$ &  &  \\
 & $3$ & $2$  \\
\hline
\end{tabular}
\end{center}
\end{table}

In order to distinguish the Betti diagrams having the simplest possible resolution for a number of steps, we recall the definition due to Green and Lazarsfeld \cite{GL1}:

\begin{definition}
 One says that a polarized variety $(X,L)$ satisfies property $(N_p)$ if it is projectively normal and $b_{j,q}(X,L)=0$ for $j\leq p $ and $q\geq 2$.
\end{definition}

In other words, a variety $\varphi_L:X\hookrightarrow \PP^r$ has property $(N_1)$ if it is projectively normal and its ideal $I_{X/\PP^r}$ is generated by quadrics $q_1, \ldots, q_s$. The number of these
quadrics is equal to $s:=b_{1,1}(X,L)={r+2\choose2}-h^0(X,L^{2})$ and is thus determined by the numerical characters of $X$. We say that $(X,L)$ has property $(N_2)$ if all of the above hold
and, in addition, all the syzygies between these quadrics are generated by linear syzygies of the type
$\ell_1 q_1+\cdots+\ell_{s} q_{s}=0$, where the $\ell_i$ are \emph{linear forms}.

\vskip 4pt

Castelnuovo, using what came to be referred to as the \emph{Base Point Free Pencil Trick}, has proven  that if $L$ is a line bundle of degree $\mbox{deg}(L)\geq 2g+2$ on a smooth curve $C$ of genus $g$,
then the curve embedded by the complete linear system $|L|$ is projectively normal and its ideal is generated by quadrics. In other words, using modern terminology, it verifies property $(N_1)$. This fact has been generalized by Green
\cite{G} to include the case of higher syzygies as well. This result illustrates a general philosophy that at least for curves, the more positive a line bundle is, the simpler its syzygies are
up to an order that is linear in the degree of the line bundle.

\begin{thm}\label{greeenvanish} Let $L$ be a line bundle of degree $d\geq 2g+p+1$ on a smooth curve $C$ of genus $g$. Then $C$ verifies property $(N_p)$.
\end{thm}
\begin{proof}
Using the description of Koszul cohomology given in Proposition \ref{kernelbundles},  we have the equivalence
 $$K_{p,2}(C,L)=0\Longleftrightarrow H^1\bigl(C, \bigwedge^{p+1}M_L\otimes L\bigr)=0.$$
Denoting by $Q_L:=M_L^{\vee}$, by Serre duality this amounts  to $H^0\bigl(C, \bigwedge^{p+1}Q_L\otimes \omega_C\otimes L^{\vee}\bigr)=0$.

\vskip 4pt

To establish this vanishing, we use a filtration on the vector bundle $M_L$ used several times by Lazarsfeld, for instance in \cite{La2} Lemma 1.4.1. We choose general points $p_1, \ldots, p_{r-1}\in C$, where $r:=r(L)=d-g$.
Then by induction on $r$ one has the following exact sequence on $C$:
$$0\longrightarrow M_{L(-p_1-\cdots-p_{r-1})}\longrightarrow M_L\longrightarrow \bigoplus_{i=1}^{r-1} \OO_C(-p_i)\longrightarrow 0.$$
Noting that $L(-p_1-\cdots-p_{r-1})$ is a pencil, by the Base Point Free Pencil Trick, one has the identification
$$M_{L(-p_1-\cdots-p_{r-1})}\cong L^{\vee}(p_1+\cdots+p_{r-1}).$$ Thus by dualizing, the exact sequence above becomes:
$$0\longrightarrow \bigoplus_{i=1}^{r-1}\OO_C(p_i)\longrightarrow Q_L\longrightarrow L(-p_1-\cdots-p_{r-1})\longrightarrow 0.$$
Taking $(p+1)$-st exterior powers in this sequence and tensoring with $\omega_C\otimes L^{\vee}$, we obtain that the vanishing
$H^0\bigl(C, \bigwedge^{p+1}Q_L\otimes \omega_C\otimes L^{\vee}\bigr)=0$
holds, once we establish that for each subdivisor $D_{p+1}$ of degree $p+1$ of the divisor $p_1+\ldots+p_{r-1}$, one has
$$H^0(C,\omega_C\otimes L^{\vee}(D_{p+1}))=0,$$ and for each subdivisor $D_{r-1-p}$ of $p_1+\cdots+p_{r-1}$ one has $$H^0(C,\omega_C(-D_{r-1-p}))=0.$$ The first
vanishing follows immediately for degree reasons, the second is implied by the inequality $r-1-p=d-g-p-1\geq p$, which is precisely
our hypothesis.
\end{proof}

The same filtration argument on the vector bundle $M_L$ shows that in the case of curves, the Betti diagram consists of only two
rows, namely that of linear syzygies and that of quadratic syzygies respectively.

\begin{prop}
 Let $L$ be a globally generated non-special line bundle $L$ on a smooth curve $C$. Then $K_{p,q}(C,L)=0$ for all $q\geq 3$ and all $p$.
\end{prop}

Thus making abstraction of the $0$-th row in which the only non-zero entry is $b_{0,0}=1$, the resolution of each non-special curve $C\subseteq \PP^r$
has the following shape:

\begin{table}[htp!]
\begin{center}
\begin{tabular}{|c|c|c|c|c|c|c|c|c|}
\hline
$1$ & $2$ & $\ldots$ & $p-1$ & $p$ & $p+1$   & $\ldots$ & $r$\\
\hline
$b_{1,1}$ & $b_{2,1}$ & $\ldots$ & $b_{p-1,1}$ & $b_{p,1}$ & $b_{p+1,1}$  &  $\ldots$ & $b_{r,1}$ \\
\hline
$b_{1,2}$ & $b_{2,2}$  & $\ldots$ & $b_{p-1,2}$ & $b_{p,2}$ & $b_{p+1,2}$  & $\ldots$ & $b_{r,2}$\\
\hline
\end{tabular}
\end{center}

\caption{The Betti table of a non-special curve of genus $g$}
\end{table}

\vskip 4pt

As already pointed out, the Hilbert function $h_C(t)=dt+1-g$ of the curve $C$ determines the difference of Betti numbers of each diagonal in the Betti diagram.

\begin{thm}\label{diagbetti}
The difference of  Betti numbers of each diagonal diagonal of the Betti diagram of a non-special line bundle $L$ on a curve $C$ is an Euler characteristic of a vector bundle on $C$. Precisely,
\begin{equation}\label{diffbetti}
b_{p+1,1}(C,L)-b_{p,2}(C,L)=(p+1)\cdot {d-g\choose p+1}\Bigl(\frac{d+1-g}{p+2}-\frac{d}{d-g}\Bigr).
\end{equation}
\end{thm}

\vskip 4pt

Green's Theorem \label{greeenvanish} ensures that $b_{p,2}(C,L)=0$ as long $p\leq d-2g-1$. On the other hand formula (\ref{diffbetti}) indicates that when
$p$ is relatively high, then  $b_{p,2}>b_{p+1,1}(C,L)\geq 0$, that is, there will certainly appear $p$-th order non-linear syzygies. One can now distinguish two
main goals concerning syzygies of curves:

\vskip 3pt

\begin{enumerate}
\item Given a line bundle $L$ of degree $d$ on a genus $g$ curve, determine which Betti numbers are zero.
\item For those Betti numbers $b_{p,1}(C,L)$ and $b_{p,2}(C,L)$ which are non-zero, determine their exact value.
\end{enumerate}

As we shall explain, there is a satisfactory answer to the first question in many important situations, both for the linear and for the quadratic row of syzygies.
The second question remains to date largely unanswered.

\vskip 5pt

A crucial aspect of syzygy theory is to identify the geometric sources of syzygies and obtain in this ways guidance as to which Koszul cohomology groups  vanish.
The first, and in some sense most important instance of such a phenomenon, when non-trivial geometry implies non-trivial syzygies is given by
the \emph{Green-Lazarsfeld Non-Vanishing Theorem}
\cite{G}.

\begin{thm}\label{glnonvan}
 Let $L$ be a globally generated line bundle on a variety $X$ and suppose one can write $L=L_1\otimes L_2$, where $r_i:=r(L_i)\geq 1$. Then
 $$K_{r_1+r_2-1,1}(X,L_1\otimes L_2)\neq 0.$$
\end{thm}
\begin{proof}
 We follow an argument due to Voisin \cite{V2}. We choose general sections $\sigma\in H^0(X,L_1)$ and $\tau\in H^0(X,L_2)$ and introduce the vector space
 $$W:=H^0(X,L)/\bigl(\sigma\cdot H^0(X,L_2)+\tau\cdot H^0(X,L_1)\bigr).$$
 We then have the following  short exact sequence:
 $$0\longrightarrow M_{L_1}\oplus M_{L_2}\longrightarrow M_L\longrightarrow W\otimes \OO_X\longrightarrow 0.$$
 Take $(r_1+r_2-1)$-st exterior powers in this short exact sequence and obtain an injection
 \begin{equation}\label{inj1}
 \bigwedge^{r_1+r_2-1} (M_{L_1}\oplus M_{L_2})\hookrightarrow \bigwedge^{r_1+r_2-1} M_L.
 \end{equation}
 Since $\bigwedge^{r_1-1} M_{L_1}\cong Q_{L_1}\otimes L_1^{\vee}$ and $\bigwedge^{r_2-1} M_{L_2}\cong Q_{L_2}\otimes L_2^{\vee}$, whereas clearly $\bigwedge^{r_1} M_{L_1}\cong L_1^{\vee}$ and
 $\bigwedge^{r_2} M_{L_2}\cong L_2^{\vee}$, by tensoring the injection (\ref{inj1}) with
 $L\cong L_1\otimes L_2$, we obtain the following injection $Q_{L_1}\oplus Q_{L_2}\hookrightarrow \bigwedge^{r_1+r_2-1} M_L\otimes L$, leading to an injection at the level of global sections
 \begin{equation}\label{inj2}
H^0(X,L_1)^{\vee}\oplus H^0(X,L_2)^{\vee}\hookrightarrow H^0\bigl(X,\bigwedge^{r_1+r_2-1} M_L\otimes L\bigr).
\end{equation}
We recall the description of the Koszul cohomology group $$K_{r_1+r_2-1,1}(X,L)\cong H^0\bigl(X, \bigwedge^{r_1+r_2-1} M_L\otimes L\bigr)/\bigwedge^{r_1+r_2} H^0(X,L).$$
The proof will be complete once one shows that at least one element of $H^0\bigl(X,\bigwedge^{r_1+r_2-1} M_L\otimes L\bigr)$ produced via the injection (\ref{inj2}), does not lie in the image
of $\bigwedge^{r_1+r_2} H^0(X,L)$.

\vskip 4pt

In order to achieve this, let us choose a basis $(\sigma_0=\sigma, \sigma_1, \ldots, \sigma_{r_1})$ of $H^0(X,L_1)$ and  a basis $(\tau_0=\tau, \tau_1, \ldots, \tau_{r_2})$ of $H^0(X,L_2)$
respectively. Then one syzygy
obtained by the inclusion (\ref{inj2}) which gives rise to a non-zero element in $K_{r_1+r_2-1}(X,L)$ is given by the following explicit formula, see also \cite{V2} equation (1.11.1):
$$\sum_{i=0}^{r_1}\sum_{j=1}^{r_2} (-1)^{i+j} (\tau_0 \sigma_0)\wedge \ldots \wedge \widehat{(\tau_0\sigma_i)}\wedge \ldots \wedge(\tau_0\sigma_{r_1})\wedge(\sigma_0\tau_1)\wedge \ldots
\wedge \widehat{(\sigma_0\tau_j)}\wedge \ldots \wedge (\sigma_0\tau_{r_2})\otimes (\sigma_i\tau_j).$$
Under the  isomorphisms $H^0(X,L_1)^{\vee}\cong \bigwedge^{r_1} H^0(X,L_1)$ and $H^0(X,L_2)^{\vee}\cong \bigwedge^{r_2} H^0(X,L_2)$, under the injection (\ref{inj2}), this last syzygy corresponds to the element $(0, \tau_1\wedge \ldots \wedge \tau_{r_2})$.
\end{proof}

\begin{remark}
 Particularly  instructive is the case when $r_1=r_2=1$, that is, both $L_1$ and $L_2$ are pencils. Keeping the notation above, the non-zero syzygy provided by the Green-Lazarsfeld Non-Vanishing Theorem
 has the following form:
 $$\gamma:=-(\tau_0\sigma_1)\otimes (\sigma_0\tau_1)+(\tau_0\sigma_0)\otimes (\sigma_1\tau_1)\in H^0(X,L)\otimes H^0(X,L),$$
 giving rise to a non-zero element in $K_{1,1}(X,L)$. The geometric interpretation of this syzygy is transparent. The map
 $\varphi_L$ factor through the map $(\varphi_{L_1},\varphi_{L_2}):X\dashrightarrow \PP^1\times \PP^1$ induced by the two pencils $L_1$ and $L_2$. The quadric $\gamma$
 is then the pull-back of the rank $4$ quadric $\PP^1\times \PP^1\hookrightarrow \PP^3$ under the projection $\PP^r\dashrightarrow \PP^3$ induced by the space $W$.
\end{remark}

\section{Green's Conjecture}

Mark Green's Conjecture formulated in 1984 in \cite{G} is an elegant and deceptively simple statement concerning the syzygies of a canonically embedded curve
$C$ of genus $g$. Despite the fact that it has generated a lot of attention and that many important results have been established,  Green's Conjecture, in its maximal
generality, remains open. Especially in the 1980's and 1990's, progress on Green's Conjecture guided much of both the research on syzygies of algebraic varieties,
as well as the development of the computer algebra system \emph{Macaulay}.

\vskip 3pt

Suppose $C\subseteq \PP^{g-1}$ is a non-hyperelliptic canonically embedded curve of genus $g$. Our main goal is to determine the Betti numbers of $C$, in particular
understand in what way the geometry of $C$ influences the shape of the Betti diagram of its canonical embedding. As already pointed out, the Betti diagram of a curve embedded
by a non-special line bundle has only two non-trivial rows, corresponding to linear and quadratic syzygies respectively. For canonical curves the situation is
even simpler, for the two rows contain the same information.

\begin{prop}
 One has the following duality for the Koszul cohomology groups of a canonical curve:
 $$K_{p,q}(C,\omega_C)\cong K_{g-p-2,3-q}(C, \omega_C)^{\vee}.$$
\end{prop}
\begin{proof}
 Follows via Serre duality for vector bundles, using the description of Koszul cohomology in terms of Lazarsfeld bundles provided in Proposition \ref{kernelbundles}.
\end{proof}

Setting $b_{p,q}:=b_{p,q}(C,\omega_C)$, we observe that, in particular, $b_{g-2,3}=b_{0,0}=1$. Unlike in the case of non-special curves,
the Betti diagram of a canonical curve has a unique non-trivial entry in the third row. Furthermore, $b_{p,1}=b_{g-2-p,2}$, that is, the row of quadratic syzygies
is a reflection of the linear strand.

\vskip 4pt

We now apply the Green-Lazarsfeld Non-Vanishing Theorem to the case of canonical curves. We write
$$K_C=L+(K_C-L),$$
where we assume that $r:=r(L)=h^0(C,L)-1\geq 1$ and $r(K_C-L)\geq 1$. By Riemann-Roch,  $r(K_C-L)=g+r-d-1$, therefore Theorem \ref{glnonvan} implies that the following
equivalent non-vanishing statements hold:
$$K_{g+2r-d-2,1}(C,\omega_C)\neq 0\Longleftrightarrow K_{d-2r,2}(C,\omega_C)\neq 0.$$

Thus existence of linear series of type $\mathfrak g^r_d$ always leads to a non-linear syzygy of order $d-2r$.
This leads us to the definition of the \emph{Clifford index}, as a way of measuring the complexity of a curve of fixed genus.

\begin{definition}
 Let $C$ be a smooth curve of genus $g$. We define the \emph{Clifford index} of $C$ as the following quantity:
 $$\mathrm{Cliff}(C):=\mathrm{min}\Bigl\{d-2r: L\in \mathrm{Pic}^d(C), \ \ d=\mathrm{deg}(L)\leq g-1,\  \ r:=r(L)\geq 1 \Bigr\}.$$
\end{definition}

It follows from the classical Clifford Theorem that $\mbox{Cliff}(C)\geq 0$ and that $\mbox{Cliff}(C)=0$ if and only if $C$ is hyperelliptic.
The Clifford index is lower semicontinuous in families and offers a stratification of the moduli space $\cM_g$ of smooth curves of genus $g$.
The Clifford index is closely related to another important invariant, the \emph{gonality} of $C$, defined as the minimal degree of a finite map
$C\rightarrow \PP^1$. By definition $\mbox{Cliff}(C)\leq \mbox{gon}(C)-2$. It can be proved \cite{CM} that always
$$\mbox{gon}(C)-3\leq \mbox{Cliff}(C)\leq \mbox{gon}(C)-2.$$
For a general curve of fixed gonality, one has $\mbox{Cliff}(C)=\mbox{gon}(C)-2$, see \cite{CM} Corollary 2.3.2.

\vskip 4pt

It follows from general Brill-Noether theory \cite{ACGH} that for a general curve $C$ of genus $g$, one has
\begin{equation}\label{bngen}
\mbox{Cliff}(C)=\Bigl\lfloor \frac{g-1}{2}\Bigr\rfloor \  \mbox{ and } \ \mbox{gon}(C)=\Bigl\lfloor \frac{g+3}{2}\Bigr\rfloor.
\end{equation}

We can summarize the discussion above as follows:

\begin{prop}\label{greentriv}
If $C$ is a smooth curve of genus $g$ then $K_{\mathrm{Cliff}(C),2}(C,\omega_C)\neq 0.$
\end{prop}

Green's Conjecture amounts to a converse of Proposition \ref{greentriv}. By the geometric version of the Riemann-Roch Theorem, the existence of an effective  divisor $D$ of degree
$d\leq g-1$ with $h^0(C,\OO_C(D))=r+1$, amounts to the statement  that under the canonical embedding, we have  $\langle D\rangle \cong \PP^{d-r-1}\subseteq \PP^{g-1}$. Thus tautologically, special linear systems on $C$
amount to special secant planes to its canonical embedding. Under this equivalence, Green's Conjecture is saying that  the universal source of non-linear syzygies
for the canonical curve is given by secants.

\begin{conjecture}\label{greenconj}
For every smooth curve of genus $g$, one has the following equivalence:
 $$K_{p,2}(C,\omega_C)=0 \Longleftrightarrow p<\mathrm{Cliff}(C).$$
\end{conjecture}

\vskip 3pt

In light of Proposition \ref{greentriv}, the non-trivial part of Green's Conjecture is the establishing of the following vanishing:
$$K_{p,2}(C,\omega_C)=0, \ \mbox{ for all } p<\mathrm{Cliff}(C).$$
The appeal of Green's Conjecture lies in the fact that it allows one to read off the complexity of the curve (in the guise of its Clifford index), from
the equations of its canonical embedding. The Clifford index is simply the order of the first non-linear syzygy of the canonical curve $C\subseteq \PP^{g-1}$.

\vskip 4pt

Spelling out the conclusions of Conjecture \ref{greenconj}, it follows that if one denotes by $c:=\mbox{Cliff}(C)$, then the resolution of
every canonical curve of genus $g$ should have the following form:

\begin{table}[htp!]
\begin{center}
\begin{tabular}{|c|c|c|c|c|c|c|c|c|c|c|}
\hline
$0$ & $1$  & $\ldots$ & $c-1$ & $c$ & $\ldots$ & $g-c-2$ & $g-c-1$  & $\ldots$ & $g-3$ & $g-2$\\
\hline
$1$ & $0$ & $\ldots$ & $0$ & 0& $\ldots$ & 0 & 0 & $\ldots$ & 0 & 0\\
\hline
$0$ & $b_{1,1}$ & $\ldots$ & $b_{c-1,1}$ & $b_{c,1}$ & $\ldots$ & $b_{g-c-2,1}$ &  $0$& $\ldots$ & 0 & 0\\
\hline
$0$ &  $0$ & $\ldots$ & $0$ & $b_{c,2}$ &$\ldots$  &$b_{g-c-2,2}$ & $b_{g-c-1,2}$ & $\ldots$ & $b_{g-2,2}$ & 0 \\
\hline
$0$ & $0$ & $\ldots$ & $0$ & 0& $\ldots$ & 0 & 0  &  $\ldots$ & 0 & 1\\
\hline
\end{tabular}
\end{center}
    \caption{The Betti table of a canonical curve of genus $g$ and Clifford index $c$}
\end{table}

The content of Conjecture \ref{greenconj} is that $b_{j,2}=0$ for $j<c$. This can be reduced to a single vanishing, namely $b_{c-1,2}=0$.

\vskip 4pt

It has contributed to the appeal of Green's Conjecture that its first cases specialize to famous classical results in the theory of algebraic curves. When $p=0$, Conjecture
\ref{greenconj} predicts that a non-hyperelliptic canonical curve $C\subseteq \PP^{g-1}$ is projectively normal, that is, all multiplication maps
$$\mbox{Sym}^k H^0(C,\omega_C)\rightarrow H^0(C,\omega_C^{\otimes k})$$
are surjective. This is precisely the content of Max Noether's Theorem, see \cite{ACGH} page 117.

\vskip 3pt

When $p=1$, Conjecture \ref{greenconj} predicts that as long as $\mbox{Cliff}(C)>1$ (that is, $C$ is neither trigonal nor a smooth plane quintic, when necessarily $g=6$), one has
$K_{1,2}(C,\omega_C)= 0$, that is, the ideal of the canonical curve is generated by quadrics. This is the content of the Enriques-Babbage Theorem, see \cite{ACGH} page 124.

\vskip 5pt

The first non-classical case of Green's Conjecture is $p=2$, when one has to show that if $C$ is not tetragonal, then $K_{2,2}(C,\omega_C)=0$. This has been
established almost simultaneously, using vector bundle methods by Voisin \cite{V1} and using Gr\"obner basis techniques by Schreyer \cite{Sch2}. Even before that, the conjecture has been
proved for all $g\leq 8$ in \cite{Sch2}, in a paper in which many techniques that would ultimately play a major role in the study of syzygies have been introduced.

\vskip 4pt

\subsection{The resolution of a general canonical curve.}
Assume now that $C$ is a general canonical curve of odd genus $g=2i+3$. Then $\mbox{Cliff}(C)=i+1$ and in fact $C$ has a one dimensional
family $W^1_{i+3}(C)$ of pencils of degree $i+3$ computing the Clifford index. The predicted resolution of the canonical image of $C$ has the
following shape, where this time we retain only the linear and quadratic strands:

\begin{table}[htp!]
\begin{center}
\begin{tabular}{|c|c|c|c|c|c|c|c|c|}
\hline
$1$ & $2$ & $\ldots$ & $i-1$ & $i$ & $i+1$ & $i+2$  & $\ldots$ & $2i$\\
\hline
$b_{1,1}$ & $b_{2,1}$ & $\ldots$ & $b_{i-1,1}$ & $b_{i,1}$ & 0 & 0 &  $\ldots$ & 0 \\
\hline
$0$ &  $0$ & $\ldots$ & $0$ & $0$ & $b_{i+1,2}$ & $b_{i+2,2}$ & $\ldots$ & $b_{2i,2}$\\
\hline
\end{tabular}
\end{center}
    \caption{The Betti table of a general canonical curve of genus $g=2i+3$}
\end{table}
Furthermore, since in each diagonal of the Betti table precisely one entry is non-zero, we can explicitly compute all the Betti numbers
and we find:
$$b_{p,1}(C,\omega_C)=\frac{(2i+2-p)(2i-2p+2)}{p+1}{2i+2\choose p-1}, \ \ \mbox{ for } p\leq i,$$
and
$$b_{p,2}(C,\omega_C)=\frac{(2i+1-p)(2p-2i)}{p+2}{2i+2\choose p}, \ \ \mbox{ for } 2i\geq p\geq i+1.$$
Quite remarkably, in this case Green's Conjecture predicts not only which Betti numbers vanish, but also their precise value.

\vskip 4pt

Let us now move to a general curve $C$ of even genus $g=2i+2$. In this case $\mbox{Cliff}(C)=i$. This Clifford index is computed by one of the finitely many pencils
of minimal degree $i+1$. The predicted resolution of the canonical model of $C$ has the following shape:

\begin{table}[htp!]
\begin{center}
\begin{tabular}{|c|c|c|c|c|c|c|c|c|}
\hline
$1$ & $2$ & $\ldots$ & $i-1$ & $i$ & $i+1$ & $i+2$  & $\ldots$ & $2i-1$\\
\hline
$b_{1,1}$ & $b_{2,1}$ & $\ldots$ & $b_{i-1,1}$ & $b_{i,1}$ & 0 & 0 &  $\ldots$ & 0 \\
\hline
$0$ &  $0$ & $\ldots$ & $0$ & $b_{i,2}$ & $b_{i+1,2}$ & $b_{i+2,2}$ & $\ldots$ & $b_{2i-1,2}$\\
\hline
\end{tabular}
\end{center}
    \caption{The Betti table of a general canonical curve of genus $g=2i+2$}
\end{table}
Note that also in this case, in each diagonal of the Betti table precisely one entry is non-zero, which allows us to
determine the entire resolution. We find:
$$b_{p,1}(C,\omega_C)=\frac{(2i-p+1)(2i-2p+1)}{p+1}{2i+1\choose p-1}, \ \ \mbox{ for } p\leq i,$$
and

$$b_{p,2}(C,\omega_C)=\frac{(2i-p)(2p-2i+1)}{p+2}{2i+1\choose p}, \ \ \ \mbox{ for }2i-1\geq p\geq i.$$

There is a qualitative difference between the resolution of a general canonical curve of odd or even genus respectively.
In the former case the resolution is \emph{pure}, that is, at each step there are only syzygies of one given degree. In the latter case, the resolution
is not pure in the middle, for there exist both linear and quadratic syzygies of order $i$.

\subsection{The resolution of canonical curves of small genus.}
To familiarize ourselves with the content of Green's Conjecture, we concentrate on small genus and we begin with the case $g=5$. In this case we distinguish two cases, depending on whether
the curve $C$ is trigonal or not. Note that the trigonal locus $\cM_{5,3}^1$ is an effective divisor on $\cM_5$.

\begin{table}[htp!]
\begin{center}
\begin{tabular}{|c|c|c|c|}
\hline
$1$ &  &  & \\
\hline
 & $3$ &  &  \\
\hline
 & & 3& \\
 \hline
 & & & 1\\
\hline
\end{tabular}
\end{center}
    \caption{The Betti table of a non-trigonal curve of genus $g=5$}
\end{table}

From the Enriques-Babbage Theorem, if $C$ is not trigonal, then it is a complete intersection of three quadrics $Q_1,Q_2, Q_3\subseteq \PP^4$.
The resolution of $C$ has the form:
$$0\longleftarrow \Gamma_C(\omega_C)\longleftarrow S\longleftarrow S(-2)^{\oplus 3}\longleftarrow S(-4)^{\oplus 3}\longleftarrow S(-5)\longleftarrow 0.$$
The syzygies between the three quadrics are quadratic and of the trivial type $Q_i\cdot Q_j-Q_j\cdot Q_i$, for $i\neq j$.

\vskip 4pt

Assume now that $C$ is trigonal. Then it turns out that the resolution has the following shape:

\begin{table}[htp!]
\begin{center}
\begin{tabular}{|c|c|c|c|}
\hline
$1$ &  &  & \\
\hline
 & $3$ & 2 &  \\
\hline
 &2 & 3& \\
 \hline
 & & & 1\\
\hline
\end{tabular}
\end{center}
    \caption{The Betti table of a trigonal curve of genus $g=5$}
\end{table}

\vskip 4pt

The interesting feature of this table is that $b_{1,2}(C,\omega_C)=b_{2,1}(C,\omega_C)=2$. To give a geometric explanation of this fact, we recall that
the linear system on $C$ residual to the degree $3$ pencil, realizes $C$ as a plane quintic with one node. Concretely, let $X:=\mbox{Bl}_q(\PP^2)$ be the Hirzebruch surface $\mathbb F_1$
and denote by $h\in \mbox{Pic}(X)$ the pull-back of the line class and by $E\in \mbox{Pic}(X)$ the exceptional divisor respectively. The image of $X$ under the linear system $|2h-E|$
realizes $X\subseteq \PP^4$ as a cubic scroll (Observe that $(2h-E)^2=3$). As we already mentioned  $C\equiv 5h-2E\in \mbox{Pic}(X)$. From the adjunction formula, $K_C\equiv \OO_C(1)$, that is,
we have the following inclusions for the canonical curve
of genus $5$:
$$C\subseteq X\subseteq \PP^4.$$

\vskip 5pt

From the Enriques-Babbage Theorem it follows that the intersection of the three quadrics containing $C$ is precisely the cubic scroll $X$, that is $I_C(2)\cong I_X(2)$. It follows
that one  has an inclusion $K_{2,1}(X,\OO_X(1))\subseteq K_{2,1}(C,\omega_C)$.
Note that
$$K_{2,1}(X,\OO_X(1))=\mbox{Ker}\Bigl\{ I_X(2)\otimes H^0(X,\OO_X(1))\rightarrow I_X(3)\Bigr\}.$$ By direct calculation we find that $\mbox{dim } I_X(3)=13$. Since
$\mbox{dim } I_X(2)=3$, it follows that $b_{2,1}(X,\OO_X(1))\geq 2$, therefore $b_{2,1}(C,\omega_C)\geq 2$. It can now be easily showed that one actually has equality, that is, $b_{2,1}(C,\omega_C)
=2$.

\vskip 3pt
To summarize, we have a \emph{set-theoretic} equality of divisors on $\cM_5$:
$$\Bigl\{[C]\in \cM_5: K_{2,1}(C,\omega_C)\neq 0\Bigr\}= \cM_{5,3}^1.$$

\vskip 5pt

Let us now move on to the case of curves of genus $7$. In some sense this is the first non-trivial case, for $7$ is the first genus when a general canonical curve is no
longer a complete intersection. If $C$ is not $4$-gonal, Green's Conjecture predicts that $b_{2,2}(C,\omega_C)=b_{3,1}(C,\omega_C)=0$, and $C$ admits the following resolution:
\begin{table}[htp!]
\begin{center}
\begin{tabular}{|c|c|c|c|c|c|}
\hline
$1$ &  &  & & &\\
\hline
& $10$ & 16 & & &\\
\hline
 & & & 16& 10 & \\
 \hline
 & & & & & 1\\
\hline
\end{tabular}
\end{center}
    \caption{The Betti table of a non-tetragonal curve of genus $g=7$}
\end{table}

Here $b_{1,1}(C,\omega_C)=\mbox{dim } I_C(2)=10$ and since $\mbox{dim } I_C(3)={9\choose 3}-5(g-1)=54$, we find that
$$b_{2,1}(C,\omega_C)=\mbox{dim } \mbox{ Ker} \Bigl\{ I_C(2)\otimes H^0(C,\omega_C)\rightarrow I_C(3)\Bigr\}=16.$$

Assume now that $C$ is a general tetragonal curve of genus $7$, in particular $W^2_6(C)=\emptyset$. The fibres of the degree $4$ pencil on $C$
span a $3$-dimensional scroll $X\subseteq \PP^6$. The resolution of $C$ has the following form:
\begin{table}[htp!]
\begin{center}
\begin{tabular}{|c|c|c|c|c|c|}
\hline
$1$ &  &  & & &\\
\hline
& $10$ & 16 & 3 & &\\
\hline
 & & 3& 16& 10 & \\
 \hline
 & & & & & 1\\
\hline
\end{tabular}
\end{center}
    \caption{The Betti table of a general tetragonal curve of genus $g=7$}
\end{table}
The novelty compared to the previous case is that $b_{3,1}(C,\omega_C)=b_{2,2}(C,\omega_C)=3$. All these syzygies are induced from the $3$-dimensional scroll $X$.
Using the Eagon-Northcott complex \cite{Sch1} Section 1, one can easily show that
$b_{3,1}(X,\OO_X(1))=3.$

\vskip 4pt

The last case we treat is that when $W^2_6(C)\neq 0$, that is, $C$ admits a degree $6$ plane model. The subvariety
$$\cM_{7,6}^2:=\Bigl\{[C]\in \cM_7: W^2_6(C)\neq \emptyset\Bigr\}$$ is
an irreducible codimension $2$ subvariety of $\cM_7$. A general element $[C]\in \cM_{7,6}^2$ corresponds to a plane sextic curve with $3$ nodes. The lines
through any of these nodes induces a pencil of degree $4$ on $C$, that is, such a curve $C$ has three pencils of degree $4$. In fact, these are the only pencils
of minimal degree on $C$. The resolution of the canonical curve $C\subseteq \PP^6$ has the following shape:
\begin{table}[htp!]
\begin{center}
\begin{tabular}{|c|c|c|c|c|c|}
\hline
$1$ &  &  & & &\\
\hline
& $10$ & 16 & 9 & &\\
\hline
 & & 9& 16& 10 & \\
 \hline
 & & & & & 1\\
\hline
\end{tabular}
\end{center}
    \caption{The Betti table of a general plane sextic of genus $g=7$}
\end{table}

Intuitively, one explains the value $b_{3,1}(C,\omega_C)=9$, by referring to the previous case, of curves having a single $\mathfrak g^1_4$. Each such pencil gives a contribution
of $3$ syzygies to the Koszul cohomology group $K_{3,1}(C,\omega_C)$ and it turns out that the $3$ pencils on a general curve $[C]\in \cM_{7,6}^2$ contribute independently, leading
to $9$ independent syzygies.

\vskip 4pt

The previous two Betti tables illustrate vividly the limitations of Green's Conjecture, for they have the same shape since they correspond to curves of the same Clifford index,
yet they have different values, depending on the number of minimal pencils of the curve in question.

\vskip 3pt

\subsection{Voisin's proof of the Generic Green Conjecture.}

Without a doubt, the most important work on Green's Conjecture is Voisin's proof \cite{V3}, \cite{V4} of Green's Conjecture for generic curves of even and odd genus
respectively. Her key idea is to specialize to curves lying on a $K3$ surface and interpret the Koszul cohomology groups in question as cohomology groups
of certain tautological vector bundles on the Hilbert scheme classifying $0$-dimensional subschemes of the ambient $K3$ surface. This novel approach to Koszul cohomology also bore fruit later
(though replacing the $K3$ surface with the symmetric product of the curve)
in the sensationally simple proof of the \emph{Gonality Conjecture} by Ein and Lazarsfeld \cite{EL}, which however is not a subject of these lectures.

\vskip 4pt

We start with the case of curves of even genus  $g=2i+2$. Therefore a generic curve $[C]\in \cM_g$ has  $\mbox{Cliff}(C)=i$. Due to the semicontinuity of dimensions of Koszul cohomology
groups in flat families, in order to prove the generic Green Conjecture it suffices to exhibit a single smooth curve $C$ of \emph{maximal} Clifford index which satisfies the two equivalent vanishing statements
$$K_{i-1,2}(C,\omega_C)=0\Longleftrightarrow K_{i+1,1}(C,\omega_C)=0.$$
We now quote the main result from \cite{V3}. Explaining the proof would take us too far afield and we refer to \cite{AN} Section 6.3 for a good summary of the main points in Voisin's proof.

\begin{thm}\label{vois1}
Let $(S,H)$ be a polarized $K3$ surface with $H^2=4i+2$ and $\mathrm{Pic}(S)=\mathbb Z\cdot H$. Then
$$K_{i-1,2}(S,H)=0.$$
\end{thm}

Theorem \ref{vois1} immediately leads to a proof of Green's Conjecture for generic curves of even genus. Indeed, first we notice that a general
$C\in |H|$ is Brill-Noether-Petri general due to Lazarsfeld's result \cite{La1}. In
particular, $C$ has the maximal Clifford index, that is, $\mbox{Cliff}(C)=i$.
On the other hand, by the already mentioned Lefschetz Hyperplane Principle, one has
$$K_{i-1,2}(C,\omega_C)\cong K_{i-1,2}(S,H),$$
thus Theorem \ref{vois1} provides the vanishing required by Green's Conjecture.

\vskip 5pt

The case of generic curves of odd genus $g=2i+3$ is treated in \cite{V4}. The strategy is to specialize again to a curve $C$ lying on a $K3$ surface $S$, but this case is
harder because of an extra difficulty. A general curve of genus $2i+3$ has a one dimensional family of minimal pencils
of minimal degree $i+3$. This is in  contrast with the even genus case.  A general curve $C\in |H|$ as in Theorem \ref{vois1}, where $g(C)=2i+2$,  has a finite number of minimal pencils $A\in W^1_{i+2}(C)$.
Each such pencil induces a rank $2$ \emph{Lazarsfeld-Mukai} vector bundle $F_{A}$ on $S$, defined by an elementary transformation on the surface $S$:
$$0\longrightarrow F_A\longrightarrow H^0(C,A)\otimes \OO_S\longrightarrow A\longrightarrow 0.$$
The geometry of the bundle $F_A$ is essential in the proof of Theorem \ref{vois1} and, crucially, $F_A$ does \emph{not} depend on the choice of $A\in W^1_{i+2}(C)$, so that it is an object one can canonically attach
to the curve $C\subseteq S$. This feature no longer holds true
in odd genus, and in order to circumvent this difficulty, Voisin uses a more special $K3$ surface instead:

\begin{thm}\label{vois2}
Let $S$ be a smooth $K3$ surface such that $\mathrm{Pic}(S)\cong \mathbb Z\cdot C\oplus \mathbb Z\cdot \Delta$, where $\Delta$ is a smooth rational curve with $\Delta^2=-2$ and $C$ is
a smooth curve of genus $g=2i+3$, such that $C\cdot \Delta=2$.
Then the following hold:
$$K_{i,2}(S,C+\Delta)=0 \ \mbox{ and } \  K_{i,2}(S,C)=0.$$
\end{thm}

Note that $C+\Delta$ can be regarded  as a semi-stable curve of genus $2i+4$. The vanishing $K_{i,2}(S,C+\Delta)=0$ is what one would expect from Theorem \ref{vois1}
(although that result has been established
only for $K3$ surfaces of Picard number one, an extension of the argument shows that the statement holds in this case as well). The most difficult part of \cite{V4} is showing how one can pass
from this vanishing to the statement $K_{i,2}(S,C)=0$, which shows that $C\subseteq S$ verifies Green's Conjecture.

\vskip 5pt

\subsection{The result of Hirschowitz and Ramanan.}
We are now going to discuss the beautiful paper \cite{HR}, which predates Voisin's papers \cite{V3},\cite{V4}. Although at the time the results of Hirschowitz and Ramanan were conditional, they provided substantial evidence for Green's Conjecture. Once the Generic Green Conjecture became a theorem, the results in \cite{HR} could be used effectively to extend the range of validity for Green's Conjecture for various classes of non-generic curves.

\vskip 3pt

We fix an odd genus $g=2i+3$ and observe that the Hurwitz locus
$$\cM_{g,i+2}^1:=\Bigl\{[C]\in \cM_g: \exists C\stackrel{i+2:1}\longrightarrow \PP^1\Bigr\}$$
is an irreducible divisor on $\cM_g$. This divisor has been studied in detail by Harris and Mumford \cite{HM} in their course of proving that
$\mm_g$ is of general type for large odd genus. In particular, they determined the class of the closure $\mm_{g,i+2}^1$ in $\mm_g$ of the Hurwitz divisor in terms
of the standard generators of $\mbox{Pic}(\mm_g)$.

\vskip 5pt

On the other hand, one can consider the Koszul divisor
$$\mathfrak{Kosz}_g:=\Bigl\{[C]\in \cM_g: K_{\frac{g-3}{2},2}(C,\omega_C)\neq 0\Bigr\}.$$

We have already explained that in genus $3$, the Koszul divisor is the degeneracy locus of the map of vector bundles over $\cM_3$, globalizing the morphisms
$$\mbox{Sym}^2 H^0(C,\omega_C)\longrightarrow H^0(C,\omega_C^{2}).$$
Max Noether's Theorem, that is, Green's Conjecture in genus $3$, implies  the already discussed set-theoretic equality $\mathfrak{Kosz}_3=\cM_{3,2}^1$. We have also explained
how in genus $5$, we have the following description of the Koszul divisor
$$\mathfrak{Kosz}_5:=\Bigl\{[C]\in \cM_5: I_C(2)\otimes H^0(C,\omega_C)\stackrel{\ncong}\longrightarrow  I_C(3)\Bigr\}.$$

It is relatively standard to show that for any odd genus $\mathfrak{Kosz}_g$ is a virtual divisor, that is, the degeneracy locus of a map between vector bundles of the same
rank over $\cM_g$. Precisely, one can show that $[C]\in \mathfrak{Kosz}_g$ if and only if the restriction map
$$H^0\Bigl(\PP^{g-1}, \bigwedge^i M_{\PP^{g-1}}(2)\Bigr) \longrightarrow H^0\Bigl(C,\bigwedge^i M_{\omega_C}\otimes \omega_C^{2}\Bigr)$$
is not an isomorphism. Both vector spaces appearing above have the same dimension, independent of $C$. Therefore either $\mathfrak{Kosz}_g$ is a genuine divisor on $\cM_g$, or else, $\mathfrak{Kosz}_g=\cM_g$. Voisin's result \cite{V4} rules out
the second possibility, for $K_{i,2}(C,\omega_C)=0$, for a general curve $[C]\in \cM_g$.

\vskip 4pt

Hirschowitz and Ramanan generalized to arbitrary odd genus the equalities of cycles in moduli already discussed in small genus. Putting together their work with that of Voisin \cite{V4}, we obtain the following result,
present in a slightly revisionist fashion, for as we mentioned, \cite{HR} predates the papers \cite{V3} and \cite{V4}:

\begin{thm}\label{hirr}
For odd genus $g=2i+3$, one has the following equality of effective divisors:
 $$[\mathfrak{Kosz}_{g}]=(i+1)[\cM_{g,i+2}^1]\in \mathrm{Pic}(\cM_{2i+3}).$$
It follows that Green's Conjecture holds for every smooth curve of genus $g$ and maximal Clifford index  $i+1$. Equivalently, the following equivalence holds
for any smooth curve $C$ of genus $g$:
$$K_{i,2}(C,\omega_C)\neq 0\Longleftrightarrow \mathrm{Cliff}(C)\leq i.$$
\end{thm}

This idea bearing fruit in \cite{HR}, of treating Green's Conjecture variationally as a moduli question, is highly innovative and has been put to use in other non-trivial contexts,
for instance in \cite{F1}, \cite{F2} or \cite{FK1}.

\vskip 6pt

\subsection{Aprodu's work and other applications.} Theorem \ref{hirr}  singles out an explicit class of curves of odd genus for which Green's Conjecture is known to hold.
It is also rather clear that Theorem \ref{hirr} can be extended to  certain classes of stable curves, for instance to all \emph{irreducible} stable curves of genus $g=2i+3$, for which one still has the equivalence
$$K_{i,2}(C,\omega_C)\neq 0\Longleftrightarrow [C]\in \mm_{g,i+2}^1.$$

Indeed, the definition of $\mathfrak{Kosz}_g$ makes sense for irreducible nodal canonical
curves, whereas Harris and Mumford \cite{HM} constructed a compactification of Hurwitz spaces, and thus in particular of $\cM_{g,i+2}^1$, by means of \emph{admissible covers}.
Using such a degenerate form of Theorem \ref{hirr}, Aprodu \cite{Ap} provided a sufficient Brill-Noether type condition for curves of \emph{any} gonality which implies Green's Conjecture.

\begin{thm}\label{apr}
Let $C$ be a smooth $k$-gonal curve of genus $g\geq 2k-2$. Assume that
\begin{equation}\label{lgc}\mathrm{dim } \ W^1_{g-k+2}(C)=g-2k+2.\end{equation}
 Then $C$ verifies Green's Conjecture.
\end{thm}

Since $C$ is $k$-gonal, by adding $g-2k+2$ arbitrary base points to a pencil $A$ of minimal degree on $C$, we observe that $\{A\}+C_{g-2k+2}\subseteq W^1_{g-k+2}(C)$.
In particular, $\mbox{dim } W^1_{g-k+2}(C)\geq g-2k+2$. Thus condition (\ref{lgc}) requires that there be no more pencils of degree $g-k+2$ on $C$ than one would normally expect.
The main use of Theorem \ref{apr} is that it reduces Green's Conjecture, which is undoubtedly a very difficult question, to Brill-Noether theory, which is by comparison easier. Indeed, using
Kodaira-Spencer theory it was showed in \cite{AC} that condition (\ref{lgc}) holds for  general $k$-gonal curves of any genus. This implies the following result:

\begin{thm}\label{greengon}
 Let $C$ be a general $k$-gonal curve of genus $g$, where $2\leq k\leq \frac{g+2}{2}$. Then $C$ verifies Green's Conjecture.
\end{thm}

Theorem \ref{greengon} has been first proven without using the bound given in Theorem \ref{apr} by Teixidor \cite{T} in the range $k\leq \frac{g}{3}$ and by Voisin \cite{V3} in the range $h\geq \frac{g}{3}$. Note that in each gonality stratum $\cM_{g,k}^1\subseteq \cM_g$, Green's Conjecture amounts to a \emph{different} vanishing statement, that is, one does not have
a uniform statement of Green's Conjecture over $\cM_g$.  Theorem
\ref{greengon} has thus to be treated one gonality stratum at a time.

\vskip 4pt

The last application we mention involves curves lying on $K3$ surfaces and discusses results from the paper \cite{AF1}:

\begin{definition}
A polarized $K3$ surface of genus $g$ consists of a pair
$(S,H)$, where $S$ is a smooth $K3$ surface and $H\in \mathrm{Pic}(S)$ is an ample class with $H^2=2g-2$. We denote by $\mathcal{F}_g$ the irreducible
$19$-dimensional moduli space of polarized $K3$ surfaces of genus $g$.
\end{definition}

The highly interesting geometry of $\mathcal{F}_g$ is not a subject of these lectures. We refer instead to \cite{Dol} for a general reference.

\vskip 4pt

As already discussed, if $[S,H]\in \mathcal{F}_g$ is a general polarized $K3$ surface and $C\in |H|$  is any smooth hyperplane section of $S$, Voisin proved that $C$ verifies Green's Conjecture.
Making decisive use of Theorem \ref{apr}, one can extend this result to arbitrary polarized $K3$ surfaces. We quote from \cite{AF1}:

\begin{thm}\label{aprf}
 Green's Conjecture holds for a smooth curve $C$ lying on any $K3$ surface.
\end{thm}

Observe a significant difference between this result and Theorems \ref{vois1} and \ref{vois2}. Whereas the lattice condition on $\mbox{Pic}(S)$ in the latter cases forces that
$C\in |H|$ has maximal Clifford index $\mbox{Cliff}(C)=\lfloor \frac{g-1}{2}\rfloor$, Theorem \ref{aprf} applies to curves in every gonality stratum in $\cM_g$.

\section{The Prym-Green Conjecture}

In this Section we would like to discuss a relatively new conjecture concerning the resolution of a general paracanonical curve and we shall begin with a general definition.

\begin{definition} Let $C$ be a smooth curve of genus $g$ and $L\in \mathrm{Pic}^d(C)$ a very ample line bundle inducing an embedding
 $\varphi_L:C\hookrightarrow \PP^r$. We say that the the pair $(C,L)$ has a \emph{natural resolution}, if for every $p$ one has
 $$b_{p,2}(C,L)\cdot b_{p+1,1}(C,L)=0.$$
\end{definition}

The naturality of the resolution of $C\subseteq \PP^r$ implies that the lower bounds on the number of syzygies of $C$ given by the Hilbert function of $C$ are attained, that is,
the minimal resolution of $C$ is as simple as the degree and genus of $C$ allow it. Recall the statement of Theorem \ref{diagbetti}: When $h^1(C,L)=0$ and thus $r=d-g$, the difference in Betti numbers on each
diagonal of the Betti table does not vary with $C$ and $L$  and is given by the following formula:
\begin{equation}\label{eulerchar}
b_{p+1,1}(C,L)-b_{p,2}(C,L)=(p+1){d-g\choose p+1}\Bigl(\frac{d+1-g}{p+2}-\frac{d}{d-g}\Bigr).
\end{equation}

Thus if one knows that $b_{p+1,1}(C,L)\cdot b_{p,2}(C,L)=0$, then for any given $p$, depending on the sign of the formula appearing in the right hand side
of (\ref{eulerchar}), one can determine which Betti number has to vanish, as well as the exact value of the remaining number on the same diagonal of the Betti table.

\vskip 4pt

Using the concept of natural resolution, one obtains a very elegant and compact reformulation of Green's Conjecture for generic curves. By inspecting again  Tables 3 and 4,
Voisin's Theorems \ref{vois1} and \ref{vois2} can be summarized
in one single sentence:

\begin{thm}
The minimal resolution of a generical canonical curve of genus $g$ is natural.
\end{thm}

Note that this is the only case when Green's Conjecture is equivalent to the resolution being natural. For a curve $C$ of non-maximal Clifford index,
that is, $\mbox{Cliff}(C)\leq \lfloor \frac{g-3}{2}\rfloor$, the resolution of the canonical curve is not natural, irrespective of whether Green's Conjecture is valid for $C$ or not.
Indeed, for integers $p$ such that

$$\mbox{Cliff}(C)\leq p <g-\mbox{Cliff}(C)-2,$$
the Green-Lazarsfeld Non-Vanishing Theorem implies $b_{p,2}(C,\omega_C)\neq 0$ and $b_{p+1,1}(C,\omega_C)\neq 0$.
This observation suggests that, more generally, naturality might be suitable to capture the resolution of a \emph{general} point of a moduli space of curves, rather than that of
arbitrary objects with certain numerical properties. A concept very similar to naturality appears in the formulation \cite{FMP} of the \emph{Minimal Resolution Conjecture} for general sets of points
on projective varieties.

\vskip 4pt

\begin{definition}\label{parac}
A \emph{paracanonical} curve of genus $g$ is a smooth genus $g$  curve embedded by a linear system
$$\varphi_{\omega_C\otimes\eta}:C\hookrightarrow \PP^{g-2},$$
where $\eta \in \mathrm{Pic}^0(C)$ is a non-trivial line bundle. When $\eta$ is an $\ell$-torsion point in $\mathrm{Pic}^0(C)$ for some $\ell\geq 2$,
we refer to a \emph{level} $\ell$ \emph{paracanonical} curve.
\end{definition}

The case studied by far the most  is that of \emph{Prym canonical} curves, when $\ell=2$. Due to work of Mumford, Beauville, Clemens, Tyurin and others,
it has been known for a long time that properties of theta divisors on Prym varieties can be reformulated in terms of the projective geometry of the corresponding
Prym canonical curve. We refer to \cite{ACGH} Section 6 for an introduction to this beautiful circle of ideas.

\vskip 4pt

Pairs $[C, \eta]$, where $C$ is a smooth curve of genus $g$ and $\eta\in \mbox{Pic}^0(C)[\ell]$ is an $\ell$-torsion point form an irreducible moduli space $\cR_{g,\ell}$, which is a finite cover of
$\cM_g$. The morphism $\cR_{g,\ell}\rightarrow \cM_g$ is given by forgetting the level $\ell$ structure $\eta$.  The geometry of the moduli space $\cR_{g,\ell}$ is discussed in detail in
\cite{CEFS}.

\vskip 4pt

We now fix a general point $[C,\eta]\in \cR_{g,\ell}$, where $g\geq 5$. The paracanonical linear system $\omega_C\otimes \eta$
induces an embedding
$$\varphi_{\omega_C\otimes \eta}:C\hookrightarrow \PP^{g-2}.$$
The Prym-Green Conjecture, originally formulated for $\ell=2$ in \cite{FL} and in the general case in \cite{CEFS}, concerns the resolution of the paracanonical ring
$\Gamma_C(\omega_C\otimes \eta)$.

\begin{conjecture}\label{pg}
The resolution of a general level $\ell$ paracanonical curve of genus $g$ is natural.
\end{conjecture}

There is an obvious weakening of the Prym-Green Conjecture, by allowing $\eta$ to be a general line bundle of degree zero on $C$, rather
than an $\ell$-torsion bundle.

\begin{conjecture}\label{pggen}
The resolution of a general paracanonical curve $C\subseteq \PP^{g-2}$ of genus $g$ and degree $2g-2$ is natural.
\end{conjecture}

Since one has an embedding of $\cR_{g,\ell}$ in the universal Jacobian variety of degree $2g-2$ over $\cM_g$, clearly Conjecture \ref{pggen}
is a weaker statement than the Prym-Green Conjecture. The numerology in Conjectures \ref{pg} and \ref{pggen} is the same, one conjecture
is just a refinement of the other.

\vskip 4pt

We would like to spell-out some of he implications of the Prym-Green Conjecture.
For odd $g$, the Prym-Green Conjecture amounts to the following  independent vanishing statements:

\begin{equation}\label{pgodd}
 K_{\frac{g-3}{2},1}\bigl(C, \omega_C \otimes \eta\bigr)=0 \ \ \ \mbox{and} \ \ K_{\frac{g-7}{2},2}\bigl(C, \omega_C \otimes \eta\bigr)=0.
\end{equation}

Assuming naturality of the resolution, one is able to determine explicitly the entire resolution. Setting $g:=2i+5$, one obtains the following resolution

\begin{table}[htp!]
\begin{center}
\begin{tabular}{|c|c|c|c|c|c|c|c|c|}
\hline
$1$ & $2$ & $\ldots$ & $i-1$ & $i$ & $i+1$ & $i+2$  & $\ldots$ & $2i+2$\\
\hline
$b_{1,1}$ & $b_{2,1}$ & $\ldots$ & $b_{i-1,1}$ & $b_{i,1}$ & 0 & 0 &  $\ldots$ & 0 \\
\hline
$0$ &  $0$ & $\ldots$ & $0$ & $b_{i,2}$ & $b_{i+1,2}$ & $b_{i+2,2}$ & $\ldots$ & $b_{2i+2,2}$\\
\hline
\end{tabular}
\end{center}
    \caption{The Betti table of a general paracanonical curve of genus $g=2i+5$}
\end{table}

\vskip 3pt

where,
 $$b_{p,1}=\frac{p(2i-2p+1)}{2i+3}{2i+4\choose p+1} \ \  \mbox{ for } p\leq i,\  \mbox{ and } \ \  b_{p,2}=\frac{(p+1)(2p-2i+1)}{2i+3}{2i+4\choose p+2} \  \ \  \mbox{ for } p\geq i.$$

\vskip 3pt

The resolution is \emph{natural}, but fails to be \emph{pure} in column $i$, for both Koszul cohomology groups
$K_{i,1}(C,\omega_C\otimes \eta)$ and $K_{i,2}(C,\omega_C\otimes \eta)$ are non-zero. Note also the striking resemblance of the minimal resolution of the general level $\ell$ paracanonical curve of
\emph{odd} genus and the resolution of the general canonical curve of \emph{even} genus.

\vskip 6pt

For even genus, the Prym-Green Conjecture reduces to a single vanishing statement, namely:
\begin{equation}\label{pgeven}
K_{\frac{g-4}{2},1}\bigl(C, \omega_C \otimes \eta\bigr)=K_{\frac{g-6}{2},2}\bigl(C, \omega_C \otimes \eta\bigr)=0.
\end{equation}
Indeed, by applying (\ref{eulerchar}), one always has $b_{\frac{g-4}{2},1}(C,\omega_C\otimes \eta)=b_{\frac{g-6}{2},2}(C,\omega_C\otimes \eta)$, so the
Prym-Green Conjecture in even genus amounts to one single vanishing statement. Like in the previous case,  naturality determines the resolution completely.

\vskip 4pt

We write $g:=2i+6$ and using (\ref{diagbetti}), obtain the following table:

\begin{table}[htp!]
\begin{center}
\begin{tabular}{|c|c|c|c|c|c|c|c|c|}
\hline
$1$ & $2$ & $\ldots$ & $i-1$ & $i$ & $i+1$ & $i+2$  & $\ldots$ & $2i+3$\\
\hline
$b_{1,1}$ & $b_{2,1}$ & $\ldots$ & $b_{i-1,1}$ & $b_{i,1}$ & 0 & 0 &  $\ldots$ & 0 \\
\hline
$0$ &  $0$ & $\ldots$ & $0$ & $0$ & $b_{i+1,2}$ & $b_{i+2,2}$ & $\ldots$ & $b_{2i+3,2}$\\
\hline
\end{tabular}
\end{center}
    \caption{The Betti table of a general paracanonical curve of genus $g=2i+6$}
\end{table}

\vskip 4pt

Here $$b_{p,1}= \frac{p(i+1-p)}{i+2}{2i+5\choose p+1} \ \mbox{ for } p\leq i  \ \mbox{ and } b_{p,2}=\frac{(p+1)(p-i)}{i+2}{2i+5\choose p+2}  \ \mbox{ for } p>i.$$

In this case the resolution is both natural and pure. Therefore, the expected resolution of a general paracanonical curve
of \emph{even} genus has the same features as the resolution of a general canonical curve of \emph{odd} genus.

\vskip 4pt

The first non-trivial case of the Prym-Green Conjecture is  $g=6$, when one has to show that the multiplication map
$$\mbox{Sym}^2 H^0(C,\omega_C\otimes \eta) \longrightarrow H^0(C,\omega_C^{2} \otimes \eta^{2})$$
is an isomorphism for a generic choice of $[C,\eta]\in \cR_{6,\ell}$. Observe that $h^0(C,\omega\otimes \eta)=5$ whereas
$h^0(C,\omega_C^{2}\otimes \eta^{2})=15$, therefore both vector spaces appearing in the previous map have he same dimension
$15$.

A systematic study of the Prym-Green Conjecture has been undertaken in \cite{CEFS}. It has been proved with the use of \emph{Macaulay} that the conjecture holds for
all $g\leq 18$ and $\ell\leq 5$ with two possible exceptions
for $\ell=2$, when $g=8,16$. In those cases, it has been showed that in the case the underlying curve $C$ is a rational $g$-nodal curve, the level $2$ paracanonical curve $C\subseteq \PP^{g-2}$ has one unexpected syzygy. This finding strongly suggests that the Prym-Green Conjecture might fail for level $\ell=2$ and for genera which have strong divisibility properties by $2$. Testing with Macaulay the
next relevant case $g=24$ is at the moment out of reach. The weaker Conjecture \ref{pggen} is expected to hold for every $g$ without exceptions.

\vskip 3pt

\subsection{The Prym-Green Conjecture for curves of odd genus.}

We shall now discuss the main ideas of the proof presented in \cite{FK1} and \cite{FK2} of the Prym-Green Conjecure for paracanonical curves of odd genus. We set $g:=2i+5$. One aims to
exhibit a smooth level $\ell$ curve $[C,\eta]\in \cR_{g,\ell}$ whose Koszul cohomology satisfies the following two vanishing properties:
\begin{equation}\label{prymgrodd}
K_{i+1,1}(C,\omega_C\otimes \eta)=0 \ \ \mbox{ and } \ \ K_{i-1,2}(C,\omega_C\otimes \eta)=0.
\end{equation}

To construct such a curve, one can resort to $K3$ surfaces, but there is an extra difficulty in comparison to the usual Green's Conjecture, because one should produce both a
general curve of genus $g$, as well as a distinguished $\ell$-torsion point in its Jacobian. Since this point has to be sufficiently explicit
to be able to compute the Koszul cohomology of the corresponding paracanonical line bundle, it is natural to attempt to realize it as the restriction
of a line bundle defined on the ambient $K3$ surface. This program has been carried out using special $K3$ surfaces, which are  of \emph{Nikulin} type when $\ell=2$, or of \emph{Barth--Verra}
type for high $\ell$. The final solution to the Prym-Green Conjecture in odd genus is presented in the paper \cite{FK3}, using certain elliptic ruled surface which are limits of polarized $K3$
surfaces.

\vskip 5pt

We let $S$ be a smooth $K3$ surface having the following Picard lattice  $$\mbox{Pic}(S)=\mathbb Z\cdot L\oplus \mathbb Z\cdot H,$$ where $L^2=L\cdot H=2g-2$ and $H^2=2g-6$. For each smooth curve
$C\in |L|\cong \PP^g$, the restriction $\OO_C(H)$ is a line bundle of degree $2g-2$, hence $\OO_C(H-C)\in \mbox{Pic}^0(C)$. Since the Jacobian of $C$ has the same dimension $g$
as the linear system $|L|$, it is natural to expect
that there will exist finitely many curves $C\in |L|$ such that $$\eta_C:=\OO_C(H-C)=\OO_C(H)\otimes \omega_C^{\vee}\in \mbox{Pic}^0(C)$$
is an $\ell$-torsion point. A priori, one is not sure that these curves are smooth or even nodal.
We name such surfaces  after \emph{Barth--Verra}, for they were first introduced in the beautiful paper \cite{BV}.
In \cite{FK2}, we show that for the general Barth--Verra surface, the expectation outlined above can be realized.

\begin{thm} \label{generic-transversal}
For a general Barth--Verra  surface $S$ of genus $g \geq 3$ and an integer $\ell$,  there exist precisely $$\displaystyle{ 2\ell^2-2 \choose g }$$  curves $C \in |L|$ such that
$\eta_{C}^{\otimes \ell}\cong \OO_C$.
All such curves $C$ are smooth and irreducible. The number of curves $C$ such that $\eta_{C}$ has order exactly $\ell$ is strictly positive.
\end{thm}

With this very effective tool at hand, we can now state the last result we wish to discuss in these lectures:

\begin{thm}
 The Prym-Green Conjecture holds for paracanonical level $\ell$ curves of odd genus.
\end{thm}

This result is proved for $\ell=2$ in \cite{FK1} and for $\ell\geq \sqrt{\frac{g+2}{2}}$ in \cite{FK2}. It is this last proof that we shall discuss in what follows.
For the remaining levels, the  proof of the Prym-Green Conjecture is completed in \cite{FK3} using elliptic surfaces.

\vskip 4pt

Assume we are in the situation of Theorem \ref{generic-transversal} and we take a Barth--Verra $K3$ surface $S$, together with smooth curves $C\in |L|$ having genus $2i+5$ and
$D\in |H|$ having genus $2i+3$ respectively. The $\ell$-torsion point on $\eta_C=\OO_C(H-C)$ is obtained as a restriction from the surface.

\vskip 4pt

To get a grip on the Koszul cohomology groups which according to (\ref{prymgrodd}) should vanish, we first use the functoriality of Koszul cohomology and write-down the following exact
sequence:

\begin{equation}\label{grseq}
\cdots \longrightarrow K_{p,q}(S,H)\longrightarrow K_{p,q}(C,H_C)\longrightarrow K_{p-1,q+1}(S,-C,H)\longrightarrow \cdots,
\end{equation}
where, we recall that the mixed Koszul cohomology group $K_{p-1,q+1}(S,-C,H)$ is computed by the following part of the Koszul complex:
$$\cdots \longrightarrow \bigwedge^p H^0(S,H)\otimes H^0(S,qH-C)\stackrel{d_{p,q}}\longrightarrow \bigwedge^{p-1} H^0(S,H)\otimes H^0(S, (q+1)H-C)\stackrel{d_{p-1,q+1}}\longrightarrow
$$
$$\stackrel{d_{p-1,q+1}}\longrightarrow \bigwedge^{p-2} H^0\bigl(S,H)\otimes H^0(X, (q+2)H-C)\longrightarrow \cdots.$$

 Using this sequence  for $(p,q)=(i+1,1)$ and $(p,q)=(i-1,2)$ respectively, in order to prove the Prym-Green Conjecture for $g=2i+5$, it suffices to show
\begin{equation}\label{van1}
K_{i+1,1}(S,H)=0 \  \mbox{ and } \ \ K_{i-1,2}(S,H)=0,
\end{equation}
respectively
\begin{equation}\label{van2}
K_{i,2}(S,-C,H)=0 \ \mbox{ and } \ K_{i-2,3}(S,-C,H)=0.
\end{equation}

Via the Lefschetz Hyperplane Principle for Koszul cohomology, vanishing (\ref{van1}) lies in the regime governed by the classical Green's Conjecture, which has been proved in \cite{AF1} for
(curves lying on) \emph{all} smooth $K3$ surfaces, in particular for Barth-Verra surfaces as well. Since $D$ is a curve of genus $g-2=2i+3$, if one shows  that the Clifford index of $D$ is maximal, that is,
$\mbox{Cliff}(D)=i+1$, then by Theorem \ref{hirr} it follows
$$K_{i-1,2}(D,\omega_D)\cong K_{i-1,2}(S,H)=0 \  \ \mbox{ and } \ \ K_{i+1,1}(S,H)\cong K_{i+1,1}(D,\omega_D)=0.$$
The fact that $\mbox{Cliff}(D)=i+1$ amounts to a simple lattice-theoretic caclulation in $\mbox{Pic}(S)$. Once this is carried out, we conclude that (\ref{van1}) holds.

\vskip 4pt

In order to prove (\ref{van2}), one  restricts the Koszul cohomology on the surface to a general curve $D\in |H|$ to obtain isomorphisms:
$$K_{i,2}(S,-C,H)\cong K_{i,2}(D, -C_D, \omega_D) \ \mbox{ and } \ K_{i-2,3}(S, -C,H)\cong K_{i-2,3}(D, -C_D, \omega_D).$$
Via the description of twisted Koszul cohomology in terms of Lazarsfeld bundles given in Proposition \ref{kernelbundles}, we  obtain the following isomorphisms:
$$K_{i,2}(D, -C_D,\omega_D)\cong H^0\Bigl(D, \bigwedge ^i M_{\omega_D}\otimes (2K_D-C_D)\Bigr) \ \ \ \ \mbox{ and } $$
$$K_{i-2,3}(D,-C_D, \omega_D)\cong H^1\Bigl(D, \bigwedge^{i-1} M_{\omega_D}\otimes (2K_D-C_D)\Bigr)\cong H^0\Bigl(D, \bigwedge ^{i-1} M_{\omega_D}^{\vee}\otimes (C_D-K_D)\Bigr)^{\vee}.$$
Although at first sight, these new cohomology groups look opaque, showing  that they vanish is easier to prove than the original Green's Conjecture. One specializes
the Barth-Verra surface further until both
$C$ and $D$ become hyperelliptic curves. Precisely, we specialize $S$ to a $K3$ surface having the following lattice, with respect to an ordered basis
$(L, \eta, E)$, where $\eta=H-L$.

\[ \left( \begin{array}{ccc}
4i+8 & 0 & 2 \\
0 & -4  &0 \\
2 & 0 & 0
\end{array} \right).\]

Note that $E$ is an elliptic pencil cutting out a divisor of degree $2$ both on a curve $C\in |L|$ and on a curve $D\in |H|$. The curve $D$ being hyperelliptic, the Lazarsfeld
bundle $M_{\omega_D}$ splits as a direct sum of degree $2$ line bundles. The vanishing (\ref{van2}) becomes
the statement that a certain line bundle of degree $\leq g(D)-1$ on $D$ has no sections, which can be easily verified. Full details can be found in \cite{FK1} and \cite{FK2}.

\vskip 4pt

\subsection{The failure of the Prym-Green Conjecture in genus 8.}  The paper \cite{CFVV} is dedicated to understanding the failure of the Prym-Green Conjecture in genus $8$.
We choose a general Prym-canonical curve of genus $8$
$$\varphi_{\omega_C\otimes \eta}:C\hookrightarrow \PP^6,$$
set $L:=\omega_C\otimes \eta$ and  denote $I_{C,L}(k):=\mbox{Ker}\bigl\{\mbox{Sym}^k H^0(C,L)\rightarrow H^0(C,L^{k})\bigr\}$ for all $k\geq 2$. Observe that $\mbox{dim } I_{C,L}(2)=7$ and $\mbox{dim } I_{C,L}(3)=49$, therefore the map
$$\mu_{C,L}: I_{C,L}(2)\otimes H^0(C,L)\rightarrow I_{C,L}(3)$$ globalizes to a morphism of vector bundles of the \emph{same rank} over the stack $\cR_{8,2}$. In \cite{CFVV}
we present \emph{three} proofs that $\mu_{C,L}$ is never an isomorphism (equivalently $K_{2,1}(C,L)\neq 0$). We believe that one of these approaches will generalize to
higher genus and offer hints into the exceptions to  the Prym-Green Conjecture in even genus. One of the approaches, links the non-vanishing of
$K_{2,1}(C,\omega_C\otimes \eta)$ to the existence of quartic hypersurfaces in $\PP^6$ vanishing doubly along the Prym-canonical curve $C$, but which are not in the image of the
map $$\mbox{Sym}^2 I_{C,L}(2)\longrightarrow I_{C,L}(4).$$

\end{document}